\documentclass[11pt]{article}
\usepackage{amsmath,amssymb,amsthm,mathrsfs}
\usepackage{geometry}
\usepackage{hyperref}
\usepackage{enumitem}
\newtheorem{theorem}{Theorem}
\newtheorem{proposition}{Proposition}
\newtheorem{lemma}{Lemma}
\newtheorem{definition}{Definition}

\title{Finite-Energy Weak Solutions to the Quantum Isothermal Euler System via a Logarithmic Schr\"odinger Approximation}

\author{Cheng Yu\\
Department of Mathematics, University of Florida, USA\\
\texttt{chengyu@ufl.edu}}
\date{}

\begin{document}

\maketitle
\begin{abstract}
This paper investigates the collisionless quantum hydrodynamic, or quantum Euler, system in \(\mathbb{T}^3\) with the linear pressure law \(P(\rho)=\rho\). Since this pressure is associated with the logarithmic internal energy \(f(\rho)=\rho\log\rho\), the model admits a natural logarithmic Schr\"odinger approximation. By means of a regularized logarithmic Schr\"odinger equation, we rigorously construct global weak solutions to the quantum isothermal Euler system. The proof relies on the Madelung transform, the polar decomposition of the wave functions, and compactness arguments. In particular, an energy  identity is used to recover the strong convergence of the hydrodynamic variables. More broadly, the analysis provides a robust Schr\"odinger approximation framework for QHD models whose internal energy contains an isothermal component.
\end{abstract}

\newline


\medskip



\section{Introduction}

Quantum hydrodynamic (QHD) models provide a fluid way to understand quantum mechanical systems. These models arise naturally in a variety of contexts, including semiconductor theory, superfluidity, Bose-type fluids, and quantum chemistry \cite{AI,AM,ABC,Fey,Grant, HZ, J1,JM,Kad,Kha,Landau,Wy}. In their simplest form, these systems describe the evolution of the mass density \(\rho\) and current density \(J\), and combine a compressible Euler-type system with a dispersive term coming from the Bohm potential \cite{AM,Bo,DGPS,Grant}. A standard formal model reads
\begin{equation}\label{QHD equ}
\begin{cases}
\partial_t \rho + \operatorname{div}J = 0, \\[1mm]
\partial_t J + \operatorname{div}\!\left(\dfrac{J\otimes J}{\rho}\right) + \nabla P(\rho)
=
\dfrac{\hbar^2}{2}\rho\,\nabla\!\left(\dfrac{\Delta\sqrt{\rho}}{\sqrt{\rho}}\right).
\end{cases}
\end{equation}
where
$P(\rho)=\rho^{\gamma}$ 
 is the pressure law for $\gamma\geq 1$. The quantum term $
\dfrac{\hbar^2}{2}\rho\,\nabla\!\left(\dfrac{\Delta\sqrt{\rho}}{\sqrt{\rho}}\right)
$ can be described as a quantum correction to the pressure, or can be interpreted as the quantum Bohm potential. 
This term may also rewriten in the following stress form
\[
\frac{\hbar^2}{2}\rho\,\nabla\!\left(\frac{\Delta\sqrt{\rho}}{\sqrt{\rho}}\right)
=
\frac{\hbar^2}{4}\Delta\nabla\rho
-
\hbar^2\,\operatorname{div}\!\bigl(\nabla\sqrt{\rho}\otimes\nabla\sqrt{\rho}\bigr).
\]
For this term, we refer readers to more references
\cite{AntonelliMarcati2009,AntonelliMarcati2012,CarlesDanchinSaut2012,Jungel2010, JM, ST,VasseurYu2016QNS}.

\medskip


In this paper, we are interested in the collisionless quantum Euler system \eqref{QHD equ} in $\mathbb{T}^3$ with $$P(\rho)=\rho,$$
where $\mathbb{T}^3 = (\mathbb{R}/\mathbb{Z})^3$ is the three-dimensional flat torus.
Note  that the pressure is related to the internal energy density through
\[
P(\rho)=\rho f'(\rho)-f(\rho),
\]
 \(P(\rho)=\rho\) is exactly the one associated with the logarithmic internal energy \[f(\rho)=\rho\log\rho.\]

Therefore, the isothermal pressure law naturally singles out the logarithmic Schr\"odinger equation \cite{BialynickiBirulaMycielski1976,BialynickiBirulaMycielski1979,Carles2022Survey,Cazenave1983,CazenaveHaraux1980}. as the dispersive model associated with the quantum fluid under the Madelung transform \cite{Mad,CarlesDanchinSaut2012}. More precisely, the corresponding wave equation is
\begin{equation}
\label{SH equ}
i\hbar\,\partial_t\psi+\frac{\hbar^2}{2}\Delta\psi
=
(\log|\psi|^2+1)\psi,
\end{equation}
where $\psi=\psi(t,x)$ is the quantum wave function of the system, it  is a complex-valued field encoding the quantum state.
Under the Madelung transform \cite{Mad,CarlesDanchinSaut2012}
\[
\rho=|\psi|^2,
\qquad
J=\hbar\,\Im(\overline{\psi}\nabla\psi),\]
and  these formally solves the quantum Euler system \eqref{QHD equ} with pressure \(P(\rho)=\rho\).
In this sense, the logarithmic Schr\"odinger equation is the canonical wave-mechanical counterpart of the quantum isothermal Euler system \cite{Carles2022Survey,CarlesCarrapatosoHillairet2022}.
This serves as the physical motivation for our mathematical analysis in this paper.

\medskip

The logarithmic Schr\"odinger equation considered here is closely related to the one studied by Carles and Gallagher \cite{CarlesGallagher2018}. 
While the work in studies the defocusing logarithmic Schrödinger equation with initial data in the restrictive functional space
$\mathcal F(H^\alpha)\cap H^1$, $0<\alpha\le 1$,
our analysis assumes initial data in the energy space
\[
\psi_0\in H^1(\mathbb T^3),\qquad |\psi_0|^2\log |\psi_0|^2\in L^1(\mathbb T^3).
\]
Furthermore, our goals diverge. We utilize the logarithmic Schrödinger equation as an intermediate step to construct finite-energy weak solutions for the quantum hydrodynamic (QHD) system. In contrast, Carles and Gallagher \cite{CarlesGallagher2018}  focus on the Schrödinger equation itself, analyzing its Cauchy problem and large-time asymptotic behavior.


\medskip

Our goal is to construct finite-energy weak solutions to \eqref{QHD equ} with linear pressure by exploiting this Schr\"odinger structure. More precisely, starting from initial data generated by a wave function $\psi_0 \in H^1(\mathbb{T}^3)$ satisfying
\[
|\psi_0|^2 \log |\psi_0|^2 \in L^1(\mathbb{T}^3),
\]
we seek to construct a finite-energy weak solution to the corresponding quantum hydrodynamic system, and to justify the passage from the dispersive description to the hydrodynamic one.

\medskip

The main difficulty is that the logarithmic nonlinearity is singular near vacuum. Indeed, the term $\log |\psi|^2$ diverges as $|\psi|\to 0$, which makes a direct rigorous treatment of \eqref{SH equ} delicate. At the hydrodynamic level, the difficulty is compounded by the singular nature of the Bohm correction. In particular, for the collisionless quantum Euler system, the basic energy inequality alone is not sufficient to provide the compactness needed to pass to the limit in the nonlinear terms. More precisely, it does not directly yield the uniform control of $\nabla\sqrt{\rho}$ required in the Bohm term and in the quadratic structure of the system.

\medskip

This lack of compactness cannot be handled by the standard tools available in viscous models. In the quantum Navier--Stokes setting, stronger a priori estimates are often obtained through the Bresch--Desjardins entropy \cite{BD,BreschDesjardins2006}, which provides additional control beyond the basic energy inequality and, in particular, yields improved bounds on $\nabla\sqrt{\rho}$. 
Based on this mechanism, weak existence theories for quantum Navier--Stokes systems were developed by
J\"ungel \cite{J1,Jungel2010}, Jungel-Milisic \cite{JM}, ~Germain-LeFloch \cite{GermainLeFloch2016},
~Lacroix-Violet-Vasseur \cite{
LacroixVioletVasseur2018},
Vasseur-Yu \cite{VasseurYu2016QNS}  and more recently, Carles, Carrapatoso, and Hillairet \cite{CarlesCarrapatosoHillairet2022}.
\medskip

However, in the absence of viscosity, the BD entropy is no longer available. On the other hand, earlier Schr\"odinger-based constructions of finite-energy weak solutions to QHD systems with pressure laws of the form \(P(\rho)=\rho^\gamma\), \(1<\gamma\leq 3\),  which relied on Strichartz estimates for $\nabla \psi$ to recover compactness, see \cite{AntonelliMarcati2009, AntonelliMarcati2012}. Such an approach is not applicable in the logarithmic Schr\"odinger framework considered in this paper. This makes the isothermal case with logarithmic internal energy both natural and technically nontrivial.

\medskip

To overcome these difficulties, we introduce a regularized logarithmic Schr\"odinger approximation
\begin{equation}\label{eq:intro-reg}
i\hbar\,\partial_t\psi^\delta+\frac{\hbar^2}{2}\Delta\psi^\delta
=
\bigl(\log(|\psi^\delta|^2+\delta)+1\bigr)\psi^\delta,
\qquad \delta>0.
\end{equation}
The corresponding hydrodynamic observables are defined by
\[
\rho_\delta:=|\psi_\delta|^2,
\qquad
J_\delta:=\hbar\,\Im(\overline{\psi_\delta}\nabla\psi_\delta).
\]
This regularization removes the singularity at vacuum while preserving the logarithmic energy structure associated with the pressure law $P(\rho)=\rho$. In this sense, the approximation is not merely a technical device for compactness, but is naturally adapted to the quantum isothermal Euler system \cite{BaoCarlesSuTang2019}.

\medskip

The novelty of this paper lies in the construction of weak solutions to the collisionless quantum isothermal Euler system with \(P(\rho)=\rho\) through a regularized logarithmic Schr\"odinger approximation using the energy method, without relying on viscosity or on Strichartz estimates for compactness.

\medskip


 \subsection*{Main result}
 
Before stating the main theorem, we briefly recall the notions of weak solution and finite-energy weak solution.

\begin{definition}[Weak solution]
\label{def:intro-weak}
Let \(T>0\). We say that \((\rho,J)\) is a weak solution of \eqref{QHD equ} on \([0,T)\times\mathbb{T}^3\) with initial data \((\rho_0,J_0)\) if there exist functions \(\sqrt{\rho}\) and \(\Lambda\) such that
\[
\sqrt{\rho}\in L^2([0,T);H^1(\mathbb{T}^3)),
\qquad
\Lambda\in L^2([0,T);L^2(\mathbb{T}^3)),
\]
and, defining
\[
\rho:=(\sqrt{\rho})^2,
\qquad
J:=\sqrt{\rho}\,\Lambda,
\]
the following hold:
\begin{itemize}
\item for every \(\zeta \in C_c^\infty([0,T)\times\mathbb{T}^3)\),
\[
\int_0^T\!\!\int_{\mathbb{T}^3}
\bigl(\rho\,\partial_t\zeta + J\cdot\nabla\zeta\bigr)\,dx\,dt
+
\int_{\mathbb{R}^3}\rho_0(x)\zeta(0,x)\,dx
=0;
\]
\item for every \(\varphi\in C_c^\infty([0,T)\times\mathbb{T}^3;\mathbb{T}^3)\),
\begin{align*}
\int_0^T\!\!\int_{\mathbb{T}^3}
\Bigl(
J\cdot\partial_t\varphi
+\Lambda\otimes\Lambda:\nabla\varphi
+\rho\,\operatorname{div}\varphi
+\hbar^2\nabla\sqrt{\rho}\otimes\nabla\sqrt{\rho}:\nabla\varphi
-\frac{\hbar^2}{4}\rho\,\Delta(\operatorname{div}\varphi)
\Bigr)\,dx\,dt
\\
+\int_{\mathbb{T}^3}J_0(x)\cdot\varphi(0,x)\,dx
=0.
\end{align*}
\end{itemize}
\end{definition}

\begin{definition}[Finite-energy weak solution]
\label{def:intro-fews}
A weak solution \((\rho,J)\) in the sense of Definition~\ref{def:intro-weak} is called a finite-energy weak solution if, for a.e. \(t\in[0,T)\), 
the following  energy 
\[
E(t)
=
\int_{\mathbb{T}^3}
\left(
\frac{\hbar^2}{2}|\nabla\sqrt{\rho}(t,x)|^2
+
\frac12 |\Lambda(t,x)|^2
+
\rho(t,x)\log \rho(t,x)
\right)\,dx
\]
is finite.
\end{definition}

We assume that the initial data have finite global energy, namely
\[
E(\psi_0)
:=
\int_{\mathbb{T}^3}
\left(
\frac{\hbar^2}{2}|\nabla\psi_0|^2
+
|\psi_0|^2\log |\psi_0|^2
\right)\,dx
<+\infty.
\]
By the Madelung transform and the quadratic identity
\[
\hbar^2|\nabla\psi_0|^2
=
\hbar^2|\nabla\sqrt{\rho_0}|^2
+
|\Lambda_0|^2,
\]
this is equivalent to
\[
E(0)
:=
\int_{\mathbb{T}^3}
\left(
\frac{\hbar^2}{2}|\nabla\sqrt{\rho_0}|^2
+
\frac12 |\Lambda_0|^2
+
\rho_0\log \rho_0
\right)\,dx
<+\infty.
\]

Our main result is stated as follows.
\begin{theorem}\label{thm:intro-main}
Let \(\psi_0\in H^1(\mathbb{T}^3)\),  and \(
|\psi_0|^2\log|\psi_0|^2\in L^1(\mathbb{T}^3),
\)
and define
\[
\rho_0:=|\psi_0|^2,
\qquad
J_0:=\hbar\,\Im(\overline{\psi_0}\nabla\psi_0).
\]
Then, for every \(T>0\), there exists a finite-energy weak solution \((\rho,J)\) of the collisionless quantum Euler system \eqref{QHD equ} on \([0,T)\times\mathbb{T}^3\), with $P(\rho)=\rho$,  in the sense of Definitions~\ref{def:intro-weak} and \ref{def:intro-fews}, with initial data \((\rho_0,J_0)\).  Moreover, the weak solutions satisfy the energy equality
\[
E(t)=E(0), \qquad\text{ for any }  t>0.
\]

\end{theorem}

 As an intermediate step, we show that the regularized logarithmic Schr\"odinger approximation converges, as $\delta \to 0$, to a weak solution of the limiting logarithmic Schr\"odinger equation in the natural energy class (Theorem \ref{weak solution to log SE} in Section \ref{weak solution}). In the present paper,
 this result is not an end in itself, but an intermediate step toward the hydrodynamic limit.
This provides the limiting wave function. For this wave function, we later establish an energy identity. Such an identity  supplies the strong compactness needed when passing to the quantum hydrodynamic formulation.
\medskip

A key ingredient in our analysis is the polar decomposition of wave functions, following the framework developed in the QHD literature \cite{AntonelliMarcati2009, AntonelliMarcati2012}. Writing formally
\[
\psi=\sqrt{\rho}\,e^{iS},
\]
one is led to the identity
\[
\hbar^2|\nabla\psi|^2=\hbar^2|\nabla\sqrt{\rho}|^2+|\Lambda|^2,
\Lambda=\hbar\sqrt{\rho}\,\nabla S,
\]
which reveals the quadratic structure of the energy. At the level of weak solutions, the pair $(\nabla\sqrt{\rho},\Lambda)$ has better stability properties than the velocity field itself, and this structure is crucial for the compactness argument. In particular, once strong convergence of $\nabla\psi_\delta$ is obtained, the corresponding strong convergence of $\nabla\sqrt{\rho_\delta}$ and $\Lambda_\delta$ follows through the polar decomposition.

\medskip

 Our  proof is based on the following main ingredients. First, for each fixed \(\delta>0\), the regularized logarithmic Schr\"odinger equation is globally well posed in \(H^1(\mathbb{T}^3)\), with conservation of mass and  energy \cite{Cazenave2003,GuerreroLopezNieto2010}. Second, the resulting family of approximate solutions is sufficiently compact to recover, in the limit $\delta \to 0$, a unique weak solution of the logarithmic Schr\"odinger equation (Theorem \ref{weak solution to log SE} in Section \ref{weak solution}). We then use the  argument in \cite{Haraux1981}
to establish an energy identity for the limiting weak solution (Proposition \ref{proposition-energy equality}
 in Section \ref{weak solution}). This identity, combined with the convergence of the  nonlinear energy, yields the strong convergence of the approximate gradients,
\[
\nabla \psi^\delta \to \nabla \psi
\qquad \text{in } L^2.
\]
Therefore, by the polar decomposition and the quadratic structure, the strong convergence of the hydrodynamic variables
\[
\nabla \sqrt{\rho^\delta} \to \nabla \sqrt{\rho},
\qquad
\Lambda^\delta \to \Lambda
\qquad \text{in } L^2.
\]
This provides the key compactness mechanism for passing from the dispersive approximation to the hydrodynamic weak formulation.

\medskip

The paper is organized as follows. Section 2 recalls basic facts on the polar decomposition and the underlying quadratic structure, which are central to our analysis. In Section 3, we introduce a regularized logarithmic Schrödinger approximation and establish the main properties of its solutions. Section 4 is devoted to weak solutions of the logarithmic Schrödinger equation; in particular, we prove an energy equality for such solutions. This energy equality is then used in Section 5 to upgrade weak convergence to strong convergence $\nabla\psi^{\delta}$ in $L^2$, a key compactness ingredient. Finally, in the last section, we use the compactness framework developed in the paper to pass to the limit in the regularized logarithmic Schrödinger approximation and recover a weak solution to the QHD system.
\medskip

\section{Polar decomposition and quadratic structure}\label{sec:polar}

A key ingredient in our analysis is the use of the polar decomposition of the wave function together with the associated quadratic quantities. Thus, we briefly recall the relevant facts from \cite{AntonelliMarcati2009,AntonelliMarcati2012}. The basic idea of polar decomposition  is reminiscent of the approach introduced by Brenier in \cite{Brenier1991}. 
The quadratic feature can improve compactness and stability properties at the level of weak solutions, which is crucial in our analysis.

Let $$   \psi \in H^1(\mathbb{T}^3; \mathbb{C})   $$ be the quantum wave function.
  We define the density
\[
\rho:=|\psi|^2,
\]
is the associated probability density (which plays the role of fluid density in the hydrodynamic formulation).
Whenever \(\psi(x)\neq0\), one may write
\[
\psi(x)=\sqrt{\rho(x)}\,\phi(x),
\qquad |\phi(x)|=1.
\]
This is the polar decomposition of \(\psi\), where \(\phi\) plays the role of the phase factor. Meanwhile,  one defines the current
\[
J:=\hbar\,\Im(\overline{\psi}\nabla\psi),
\]
and the momentum variable
\[
\Lambda:=\hbar\,\Im(\overline{\phi}\nabla\psi),
\]
where \(\phi\) is any measurable choice such that
\[
\psi=\sqrt{\rho}\,\phi,
\qquad |\phi|=1 \quad \text{a.e. on } \{\rho>0\}.
\]
This is the standard polar decomposition underlying the Madelung transform \cite{Mad}.
Then \(\Lambda\) is independent of the choice of \(\phi\) on the vacuum region, and one has the identity
\[
J=\sqrt{\rho}\,\Lambda.
\]

The advantage of this formulation is that the pair \((\nabla\sqrt{\rho},\Lambda)\) enjoys better stability properties than the velocity field itself. Indeed, if one formally writes
\[
\psi=\sqrt{\rho}\,e^{iS},
\]
then
\[
\nabla\psi
=
e^{iS}\nabla\sqrt{\rho}
+
i e^{iS}\sqrt{\rho}\,\nabla S,
\]
and therefore
\[
|\nabla\psi|^2
=
|\nabla\sqrt{\rho}|^2+\rho|\nabla S|^2.
\]
Since
\[
\Lambda=\hbar\sqrt{\rho}\,\nabla S,
\]
the previous identity becomes
\begin{equation}\label{eq:quadratic-identity}
\hbar^2|\nabla\psi|^2
=
\hbar^2|\nabla\sqrt{\rho}|^2+|\Lambda|^2.
\end{equation}
This is the fundamental quadratic structure underlying the energy functional. In particular, the kinetic energy may be written as
\[
\frac{\hbar^2}{2}\int_{\mathbb{T}^3}|\nabla\psi|^2\,dx
=
\frac{\hbar^2}{2}\int_{\mathbb{T}^3}|\nabla\sqrt{\rho}|^2\,dx
+
\frac12\int_{\mathbb{T}^3}|\Lambda|^2\,dx.
\]

A rigorous version of \eqref{eq:quadratic-identity} is available at the \(H^1\)-level. More precisely, if \(\psi\in H^1(\mathbb{T}^3)\), then there exists \(\phi\in L^\infty(\mathbb{T}^3)\) with
\[
\psi=\sqrt{\rho}\,\phi,
\qquad |\phi|\le 1 \quad \text{a.e.},
\]
such that
\[
\nabla\sqrt{\rho}=\Re(\overline{\phi}\nabla\psi),
\qquad
\Lambda=\hbar\,\Im(\overline{\phi}\nabla\psi),
\]
and
\begin{equation}\label{eq:rigorous-quadratic-identity}
\hbar^2\,\Re(\partial_j\overline{\psi}\,\partial_k\psi)
=
\hbar^2\,\partial_j\sqrt{\rho}\,\partial_k\sqrt{\rho}
+
\Lambda^{(j)}\Lambda^{(k)}
\end{equation}
for every \(j,k\in\{1,2,3\}\). In particular,
\[
\hbar^2|\nabla\psi|^2=\hbar^2|\nabla\sqrt{\rho}|^2+|\Lambda|^2
\qquad\text{a.e.}
\]
This quadratic identity was proved in \cite{AntonelliMarcati2009}. 
It  is especially useful for the weak stability of solutions. Indeed, strong convergence of \(\nabla\psi^\delta\) in \(L^2\) implies strong convergence of both \(\nabla\sqrt{\rho^\delta}\) and \(\Lambda^\delta\) in \(L^2\), since
\[
\nabla\sqrt{\rho^\delta}
=
\Re(\overline{\phi^\delta}\nabla\psi^\delta),
\qquad
\Lambda^\delta
=
\hbar\,\Im(\overline{\phi^\delta}\nabla\psi^\delta),
\]
with \(|\phi^\delta|\le1\). This provides a possible approach to constructing weak solutions to QHD induced from the Schrödinger formulation.

\section{The regularized logarithmic Schr\"odinger approximation}

In this section, we introduce a regularized logarithmic Schr\"odinger approximation adapted to the quantum isothermal Euler system \eqref{QHD equ}. The regularization preserves the logarithmic energy structure associated with the pressure law \(P(\rho)=\rho\), while removing the singularity of the nonlinearity at vacuum.

\subsection{Regularized internal energy and pressure}
To avoid singularities at $\rho=0$, we regularize $f(\rho)=\rho\log\rho$ by
\begin{equation}\label{eq:fdelta}
f_\delta(\rho):=(\rho+\delta)\log(\rho+\delta),\qquad \delta>0.
\end{equation}
Then
\[
f_\delta'(\rho)=\log(\rho+\delta)+1,
\qquad
P_\delta(\rho):=\rho f_\delta'(\rho)-f_\delta(\rho)=\rho-\delta\log(\rho+\delta).
\]
In particular, $P_\delta(\rho)\to \rho$ pointwise as $\delta\to0$.  Thus, we can expect this limit passage  in some $L^p$ space later when we need to vanish $\delta$.

\subsection{The regularized logarithmic Schr\"odinger approximation}
\label{subsec:reg-logNLS}

To approximate the quantum isothermal Euler system, we consider the regularized logarithmic Schr\"odinger equation
\begin{equation}\label{eq:NLSdelta}
i\hbar\,\partial_t\psi^\delta + \frac{\hbar^2}{2}\Delta\psi^\delta
=
f_\delta'(|\psi^\delta|^2)\psi^\delta
=
\bigl(\log(|\psi^\delta|^2+\delta)+1\bigr)\psi^\delta,
\qquad
\psi^\delta(0)=\psi_0\in H^1(\mathbb{T}^3),
\end{equation}
where
\begin{equation}\label{eq:fdelta-def}
f_\delta(\rho):=(\rho+\delta)\log(\rho+\delta),
\qquad
f_\delta'(\rho)=\log(\rho+\delta)+1.
\end{equation}

The hydrodynamic observables associated with \(\psi^\delta\) are defined by
\begin{equation}\label{eq:rhoJ}
\rho^\delta:=|\psi^\delta|^2,
\qquad
J^\delta:=\hbar\,\Im\!\bigl(\psi^\delta\nabla\overline{\psi^\delta}\bigr).
\end{equation}
As we shall see below, the pair \((\rho^\delta,J^\delta)\) formally satisfies a regularized quantum hydrodynamic system with pressure law determined by \(f_\delta\).

\subsection{Formal derivation of the regularized quantum Euler system}
\label{subsec:formal-QE}

In this subsection we explain formally how \eqref{eq:NLSdelta} induces a regularized quantum Euler system. For simplicity of notation, we write
\[
\psi=\psi^\delta,
\qquad
\rho=\rho^\delta,
\qquad
J=J^\delta.
\]

\paragraph{Step 1: Continuity equation.}
We compute
\[
\partial_t\rho
=
\partial_t(\psi\overline{\psi})
=
\overline{\psi}\,\partial_t\psi+\psi\,\partial_t\overline{\psi}
=
2\Re\bigl(\overline{\psi}\,\partial_t\psi\bigr).
\]
From \eqref{eq:NLSdelta},
\[
\partial_t\psi
=
\frac{i\hbar}{2}\Delta\psi-\frac{i}{\hbar}f_\delta'(|\psi|^2)\psi.
\]
Hence
\[
\partial_t\rho
=
2\Re\left(
\overline{\psi}\left(\frac{i\hbar}{2}\Delta\psi-\frac{i}{\hbar}f_\delta'(\rho)\psi\right)
\right)
=
\hbar\,\Re(i\,\overline{\psi}\,\Delta\psi),
\]
since \(f_\delta'(\rho)\rho\) is real-valued. Using \(\Re(i z)=-\Im(z)\), we obtain
\[
\partial_t\rho=-\hbar\,\Im(\overline{\psi}\,\Delta\psi).
\]
On the other hand,
\[
\operatorname{div}\Im(\overline{\psi}\nabla\psi)
=
\Im(\nabla\overline{\psi}\cdot\nabla\psi)+\Im(\overline{\psi}\Delta\psi)
=
\Im(\overline{\psi}\Delta\psi),
\]
because \(\nabla\overline{\psi}\cdot\nabla\psi=|\nabla\psi|^2\in\mathbb{R}\). Therefore,
\[
\partial_t\rho
=
-\hbar\,\operatorname{div}\Im(\overline{\psi}\nabla\psi)
=
-\operatorname{div}J,
\]
that is,
\begin{equation}\label{eq:cont-formal}
\partial_t\rho+\operatorname{div}J=0.
\end{equation}

\paragraph{Step 2: Momentum equation in stress form.}
Differentiating \(J=\hbar\,\Im(\psi\nabla\overline{\psi})\) with respect to time and using \eqref{eq:NLSdelta}, one obtains the stress-form identity
\begin{equation}\label{eq:stress-form}
\partial_t J
=
-\hbar^2\,\operatorname{div}\Re(\nabla\psi\otimes\nabla\overline{\psi})
+\frac{\hbar^2}{4}\Delta\nabla\rho
-\rho\,\nabla\bigl(f_\delta'(\rho)\bigr).
\end{equation}

\paragraph{Step 3: Identification of the pressure.}
The pressure associated with \(f_\delta\) is defined by
\begin{equation}\label{eq:Pdelta}
P_\delta(\rho):=\rho f_\delta'(\rho)-f_\delta(\rho).
\end{equation}
Since
\[
\nabla P_\delta(\rho)=\rho\,\nabla f_\delta'(\rho),
\]
the last term in \eqref{eq:stress-form} becomes \(-\nabla P_\delta(\rho)\), and thus
\begin{equation}\label{eq:stress-form-P}
\partial_t J+\nabla P_\delta(\rho)
=
-\hbar^2\,\operatorname{div}\Re(\nabla\psi\otimes\nabla\overline{\psi})
+\frac{\hbar^2}{4}\Delta\nabla\rho.
\end{equation}
A direct computation gives
\[
P_\delta(\rho)
=
\rho\bigl(\log(\rho+\delta)+1\bigr)-(\rho+\delta)\log(\rho+\delta)
=
\rho-\delta\log(\rho+\delta).
\]

\paragraph{Step 4: Polar decomposition and quadratic flux.}
By the polar decomposition lemma of \cite{AntonelliMarcati2009}, there exists a polar factor \(\phi\), with \(|\phi|\le1\), such that
\[
\psi=\sqrt{\rho}\,\phi
\qquad\text{a.e.},
\]
and, defining
\[
\Lambda:=\hbar\,\Im(\overline{\phi}\nabla\psi),
\qquad
J=\sqrt{\rho}\,\Lambda,
\]
one has
\begin{equation}\label{eq:polar-stress}
\hbar^2\,\Re(\nabla\psi\otimes\nabla\overline{\psi})
=
\hbar^2\,\nabla\sqrt{\rho}\otimes\nabla\sqrt{\rho}
+\Lambda\otimes\Lambda.
\end{equation}
Substituting \eqref{eq:polar-stress} into \eqref{eq:stress-form-P}, we obtain
\begin{equation}\label{eq:stress-form-Lambda}
\partial_t J+\nabla P_\delta(\rho)
=
-\operatorname{div}(\Lambda\otimes\Lambda)
-\hbar^2\operatorname{div}\bigl(\nabla\sqrt{\rho}\otimes\nabla\sqrt{\rho}\bigr)
+\frac{\hbar^2}{4}\Delta\nabla\rho.
\end{equation}

\paragraph{Step 5: Bohm identity.}
Using the identity
\[
\frac{\hbar^2}{2}\rho\,\nabla\!\left(\frac{\Delta\sqrt{\rho}}{\sqrt{\rho}}\right)
=
\frac{\hbar^2}{4}\Delta\nabla\rho
-\hbar^2\operatorname{div}\bigl(\nabla\sqrt{\rho}\otimes\nabla\sqrt{\rho}\bigr),
\]
we may rewrite \eqref{eq:stress-form-Lambda} as
\begin{equation}\label{eq:qhd-momentum-final}
\partial_t J+\operatorname{div}(\Lambda\otimes\Lambda)+\nabla P_\delta(\rho)
=
\frac{\hbar^2}{2}\rho\,\nabla\!\left(\frac{\Delta\sqrt{\rho}}{\sqrt{\rho}}\right).
\end{equation}

Therefore, at least formally, the observables \((\rho^\delta,J^\delta)\) induced by \eqref{eq:NLSdelta} solve the regularized collisionless quantum hydrodynamic system
\begin{equation}\label{eq:QHD-final}
\begin{cases}
\partial_t\rho^\delta+\operatorname{div}J^\delta=0,\\[1mm]
\partial_t J^\delta+\operatorname{div}(\Lambda^\delta\otimes\Lambda^\delta)+\nabla P_\delta(\rho^\delta)
=
\dfrac{\hbar^2}{2}\rho^\delta\,\nabla\!\left(\dfrac{\Delta\sqrt{\rho^\delta}}{\sqrt{\rho^\delta}}\right),
\end{cases}
\end{equation}
where \(\Lambda^\delta\) is defined through polar decomposition and
\[
P_\delta(\rho)=\rho-\delta\log(\rho+\delta).
\]

\subsection{Well-posedness for the regularized equation}

We first recall a standard well-posedness result for the regularized logarithmic Schr\"odinger equation. Since $\delta>0$ is fixed, the singularity of the logarithm at vacuum is removed, and the equation becomes a standard semilinear Schr\"odinger equation in $H^1(\mathbb{T}^3)$. We include this result only for completeness, since it provides the starting point of the approximation scheme used below.

\begin{proposition}\label{prop:global-H1-reglogNLS}
Let $\delta>0$ and $\psi_0\in H^1(\mathbb{T}^3)$. Then, for every $T>0$, the Cauchy problem
 \begin{equation}\label{eq:reglogNLS-wp}
i\hbar\,\partial_t\psi^\delta+\frac{\hbar^2}{2}\Delta\psi^\delta
=
F_\delta(\psi^\delta),
\qquad
\psi^\delta(0)=\psi_0,
\end{equation}
where
\[
F_\delta(z):=\bigl(\log(|z|^2+\delta)+1\bigr)z,
\qquad z\in\mathbb{C}.
\]
admits  a unique solution
\[
\psi^\delta\in C([0,T];H^1(\mathbb{T}^3))
\cap C^1([0,T];H^{-1}(\mathbb{T}^3)).
\]
Moreover, the mass and the regularized energy are conserved on $[0,T]$.
\end{proposition}

\begin{proof}
The proof is standard. For fixed $\delta>0$, the nonlinearity
\[
F_\delta(z)=\bigl(\log(|z|^2+\delta)+1\bigr)z
\]
is smooth and has at most cubic growth. Therefore $F_\delta$ defines a locally Lipschitz map from $H^1(\mathbb{T}^3)$ to $H^{-1}(\mathbb{T}^3)$, and the local well-posedness in
\[
C([0,T_*);H^1(\mathbb{T}^3))\cap C^1([0,T_*);H^{-1}(\mathbb{T}^3))
\]
follows from the classical $H^1$ theory for semilinear Schr\"odinger equations; see, e.g., \cite{Cazenave2003}. The conservation of mass and regularized energy is obtained by the usual density argument. Finally, the regularized energy controls the $H^1$ norm up to the conserved $L^2$ mass, since
\[
f_\delta(\rho)=(\rho+\delta)\log(\rho+\delta)\ge -C_\delta(\rho+1),
\qquad \rho\ge0,
\]
and thus the solution extends globally in time.
\end{proof}

\section{Weak solutions to the logarithmic Schr\"odinger equation} \label{weak solution}

In this section, we analyze the limiting logarithmic Schr\"odinger equation obtained from the regularized approximation. We emphasize that this part is not intended as an independent Cauchy theory for the logarithmic Schr\"odinger equation. Rather, its role is to provide the compactness and energy tools needed later in the proof of the main theorem. 

\medskip

More precisely, starting from the globally well-posed regularized problem, we recover a weak solution of the limiting logarithmic Schr\"odinger equation in the natural energy space
\[
W(\mathbb T^3):=\left\{u\in H^1(\mathbb T^3): |u|^2\log |u|^2\in L^1(\mathbb T^3)\right\},
\]
and then establish an energy equality, which is the key ingredient for proving the strong  convergence of the approximate gradients and the associated hydrodynamic variables.

\medskip

\subsection{Global existence of weak solutions to logarithmic Schr\"odinger equation}

The first compactness step is to obtain strong convergence of \(\psi^\delta\). This follows from the uniform \(H^1\)-bound, the  \(L^2\)-bound on the nonlinearity, and the Aubin--Lions lemma.

\begin{proposition}[Compactness of \(\psi^\delta\)]\label{prop:compactness}
Assume that \(\{\psi^\delta\}_{\delta>0}\) satisfies the uniform bound
\begin{equation}\label{uniform-H1}
\sup_{\delta>0}\|\psi^\delta\|_{L^\infty(0,T;H^1(\mathbb{T}^3))}\le C_T.
\end{equation}
and 
\begin{equation}\label{assum eq:uniform-F}
\sup_{\delta>0}\|F_\delta(\psi^\delta)\|_{L^\infty(0,T;L^2(\mathbb{T}^3))}\le C_{T}.
\end{equation}
Then, up to extraction of a subsequence, there exists
\[
\psi\in L^\infty(0,T;H^1(\mathbb{T}^3))
\]
such that, 
\begin{align}
\psi^\delta &\rightharpoonup^\ast \psi
\qquad \text{in }L^\infty(0,T;H^1(\mathbb{T}^3)), \label{eq:weak-star-H1}\\
\psi^\delta &\to \psi
\qquad \text{strongly in }L^2(0,T;L^2(\mathbb{T}^3)), \label{eq:strong-L2-local}
\end{align}
and, after extraction of a further subsequence if necessary,
\begin{equation}\label{eq:ae-conv}
\psi^\delta(t,x)\to\psi(t,x)
\qquad\text{for a.e. }(t,x)\in(0,T)\times \mathbb{T}^3.
\end{equation}
\end{proposition}

\begin{proof}
The uniform bound \eqref{uniform-H1} implies that \(\{\psi^\delta\}\) is bounded in
\(L^\infty(0,T;H^1(\mathbb{T}^3))\). Hence, up to extraction of a subsequence, there exists
\(\psi\in L^\infty(0,T;H^1(\mathbb{T}^3))\) such that \eqref{eq:weak-star-H1} holds.

Next we write 
\[
i\hbar\,\partial_t\psi^\delta
=
-\frac{\hbar^2}{2}\Delta\psi^\delta + F_\delta(\psi^\delta).
\]
Since \(\psi^\delta\) is bounded in
\(L^\infty(0,T;H^1(\mathbb{T}^3))\), and \(\Delta\psi^\delta\) is bounded in
\(L^\infty(0,T;H^{-1}(\mathbb{T}^3))\). Moreover, by \eqref{assum eq:uniform-F},
\(F_\delta(\psi^\delta)\) is bounded in \(L^\infty(0,T;L^2(\mathbb{T}^3))\), hence also in
\(L^\infty(0,T;H^{-1}(\mathbb{T}^3))\). Therefore,
\[
\sup_{\delta>0}\|\partial_t\psi^\delta\|_{L^\infty(0,T;H^{-1}(\mathbb{T}^3))}\le C_{T}.
\]

Since the embedding
\[
H^1(\mathbb{T}^3)\hookrightarrow L^2(\mathbb{T}^3)
\]
is compact and
\[
L^2(\mathbb{T}^3)\hookrightarrow H^{-1}(\mathbb{T}^3)
\]
is continuous, the Aubin--Lions lemma yields \eqref{eq:strong-L2-local}. Passing to a further subsequence, we also obtain the almost everywhere convergence \eqref{eq:ae-conv}.
\end{proof}

With such compactness in Proposition~\ref{prop:compactness},  we are ready to recover a weak solution to  logarithmic Schr\"odinger equation by letting $\delta\to 0.$

\begin{theorem}
\label{weak solution to log SE}
Assume that the hypotheses of Proposition \ref{prop:compactness} hold, and let \[\psi_0\in H^1(\mathbb{T}^3)\;\;\;\text{ and }\;\; 
|\psi_0|^2\log|\psi_0|^2\in L^1(\mathbb{T}^3). 
\]
Then  \(\psi\) is a unique weak solution of
\begin{equation}\label{eq:logNLS-limit}
i\hbar\,\partial_t\psi+\frac{\hbar^2}{2}\Delta\psi
=
(\log|\psi|^2+1)\psi,\quad\quad\text{ and }\;\; \psi(0,x)=\psi_0,
\end{equation}
in the sense that
\begin{equation}\label{eq:weak-form-limit}
\int_0^T\!\!\int_{\mathbb{T}^3}
\left(
-i\hbar\,\psi\,\partial_t\overline{\varphi}
-\frac{\hbar^2}{2}\nabla\psi\cdot\nabla\overline{\varphi}
-(\log|\psi|^2+1)\psi\,\overline{\varphi}
\right)\,dx\,dt
=
i\hbar\int_{\mathbb{T}^3}\psi_0(x)\overline{\varphi(0,x)}\,dx
\end{equation}
for every \(\varphi\in C_c^\infty([0,T)\times\mathbb{T}^3)\).
\end{theorem}

\begin{proof}
By Proposition~\ref{prop:global-H1-reglogNLS}, for every \(\delta>0\) there exists a unique global solution to \eqref{eq:reglogNLS-wp} satisfying the conserved energy identity
\[
E_\delta(\psi(t))
:=
\int_{\mathbb{T}^3}
\left(
\frac{\hbar^2}{2}|\nabla\psi^\delta(t,x)|^2
+
f_\delta(|\psi^\delta(t,x)|^2)
\right)\,dx
=
E_\delta(\psi(0))<+\infty,
\]
and conserved the mass

\begin{equation}\label{mass-conservation}
\|\psi^\delta(t)\|_{L^2(\mathbb{T}^3)}
=
\|\psi_0\|_{L^2(\mathbb{T}^3)}.
\end{equation}

Note that $$ f_\delta(|\psi^\delta(t,x)|^2)=(|\psi^{\delta}|^2+\delta)\log (|\psi^\delta|^2+\delta),$$
thus we are able to find a large enough number $C>0$, such that 
\[\int_{\mathbb{T}^3}
\frac{\hbar^2}{2}|\nabla\psi^\delta(t,x)|^2
\,dx+\int_{\mathbb{T}^3}(|\psi^{\delta}|^2+\delta)\log^+ (|\psi^\delta|^2+\delta)\,dx
\leq E_{\delta}(0)+\int_{\mathbb{T}^3}(|\psi^{\delta}|^2+\delta)\log^- (|\psi^\delta|^2+\delta)\,dx.
\]

Now we need to control $\int_{\mathbb{T}^3}(|\psi^{\delta}|^2+\delta)\log^- (|\psi^\delta|^2+\delta)\,dx$.
A key observation is the classical estimate
\[
\rho \log \rho \ge -\frac{1}{e} \quad \text{for all } \rho \ge 0.
\]

This give us that, 
 \[\int_{\mathbb{T}^3}
\frac{\hbar^2}{2}|\nabla\psi^\delta(t,x)|^2
\,dx+\int_{\mathbb{T}^3}(|\psi^{\delta}|^2+\delta)\log^+ (|\psi^\delta|^2+\delta)\,dx
\leq E_{\delta}(0)+\frac{C}{e},
\]where $C$ is a constant depending on $\mathbb{T}^3$.

With help of \eqref{mass-conservation}, this implies that \begin{equation}
\label{H1 bound}
\sup_{\delta>0} \|\psi^\delta\|_{L^\infty(0,T; H^1(\mathbb{T}^3))} \le C_T,
\end{equation}
\smallskip

Since 
\[
F_\delta(\psi^\delta) := (\log(|\psi^\delta|^2 + \delta) + 1) \, \psi^\delta,
\]
we can estimate, for any $\theta > 0$,
\[
|F_\delta(\psi^\delta)| \le C |\psi^\delta| \big( (|\psi^\delta|^2 + \delta)^\theta + (|\psi^\delta|^2 + \delta)^{1-\theta} \big).
\]
Choosing $\theta = \frac{1}{2}$, we obtain
\[
|F_\delta(\psi^\delta)| \le C \big( |\psi^\delta|^2 + 1 \big).
\]

Then, using the uniform $H^1$ bound \eqref{H1 bound}, and  the Sobolev embedding $H^1(\mathbb{T}^3) \hookrightarrow L^6(\mathbb{T}^3)$, we conclude that

\begin{equation}\label{eq:uniform-F}
\sup_{\delta>0}\|F_\delta(\psi^\delta)\|_{L^\infty(0,T;L^2(\mathbb{T}^3))}\le C_{T}.
\end{equation}

Thus,   up to extraction of a subsequence, there exists a function 
\[
G\in L^2((0,T)\times \mathbb{T}^3)
\]
such that
\begin{equation}
\label{weak limit of F}
F_\delta(\psi^\delta)\rightharpoonup G
\qquad\text{weakly in }L^2((0,T)\times \mathbb{T}^3).
\end{equation}

On the other hand, the strong convergence of \(\psi^\delta\) in \(L^2((0,T)\times \mathbb{T}^3)\) implies, up to a subsequence,
\[
\psi^\delta(t,x)\to \psi(t,x)
\qquad\text{for a.e. }(t,x)\in(0,T)\times \mathbb{T}^3.
\]
Therefore,
\[
|\psi^\delta(t,x)|^2+\delta \to |\psi(t,x)|^2
\qquad\text{for a.e. }(t,x)\in(0,T)\times \mathbb{T}^3,
\]
and hence
\[
\bigl(\log(|\psi^\delta|^2+\delta)+1\bigr)\psi^\delta
\to
(\log|\psi|^2+1)\psi
\qquad\text{for a.e. }(t,x)\in(0,T)\times \mathbb{T}^3.
\]
By uniqueness of the weak limit, it follows that
\begin{equation}
\label{G formu}
G=(\log|\psi|^2+1)\psi
\qquad\text{a.e. on }(0,T)\times \mathbb{T}^3.
\end{equation}

For each \(\delta>0\), the weak formulation of \eqref{eq:reglogNLS-wp} is 
\begin{equation}\label{eq:weak-form-delta}
\int_0^T\!\!\int_{\mathbb{R}^3}
\left(
-i\hbar\,\psi^\delta\,\partial_t\overline{\varphi}
-\frac{\hbar^2}{2}\nabla\psi^\delta\cdot\nabla\overline{\varphi}
-F_\delta(\psi^\delta)\,\overline{\varphi}
\right)\,dx\,dt
=
i\hbar\int_{\mathbb{T}^3}\psi_0(x)\overline{\varphi(0,x)}\,dx
\end{equation}
for every \(\varphi\in C_c^\infty([0,T)\times\mathbb{T}^3)\).

By \eqref{H1 bound} and \eqref{eq:uniform-F},  Proposition \ref{prop:compactness} gives us the compactness of this family of solutions. In particular, 
by \eqref{eq:strong-L2-local}, we have
\[
\psi^\delta\to\psi
\qquad\text{strongly in }L^2((0,T)\times \operatorname{supp}\varphi),
\]
hence
\[
\int_0^T\!\!\int_{\mathbb{T}^3}
\psi^\delta\,\partial_t\overline{\varphi}\,dx\,dt
\to
\int_0^T\!\!\int_{\mathbb{T}^3}
\psi\,\partial_t\overline{\varphi}\,dx\,dt.
\]
Also, by \eqref{eq:weak-star-H1},
\[
\nabla\psi^\delta \rightharpoonup \nabla\psi
\qquad\text{weakly in }L^2((0,T)\times \operatorname{supp}\varphi),
\]
so
\[
\int_0^T\!\!\int_{\mathbb{T}^3}
\nabla\psi^\delta\cdot\nabla\overline{\varphi}\,dx\,dt
\to
\int_0^T\!\!\int_{\mathbb{T}^3}
\nabla\psi\cdot\nabla\overline{\varphi}\,dx\,dt.
\]
Finally, by \eqref{weak limit of F} and \eqref{G formu}, 
\[
\int_0^T\!\!\int_{\mathbb{T}^3}
F_\delta(\psi^\delta)\,\overline{\varphi}\,dx\,dt
\to
\int_0^T\!\!\int_{\mathbb{T}^3}
G\,\overline{\varphi}\,dx\,dt.
\]

Since both systems share the same initial data $\psi^{\delta}_0=\psi_0$,   we can pass to the limit in \eqref{eq:weak-form-delta} and obtain \eqref{eq:weak-form-limit}.

Using a classical argument as in \cite{CarlesGallagher2018, Haraux1981}], the uniqueness of weak solutions follows from the following lemma.

\begin{lemma}[Page 119 from~\cite{Haraux1981}]\label{control of complex}
We have
\begin{equation}
\operatorname{Im}\Bigl( z_2 \ln |z_2|^2 - z_1 \ln |z_1|^2 \Bigr)\,(\overline{z_2}-\overline{z_1})
\leq C |z_2-z_1|^2,
\qquad \forall z_1,z_2 \in \mathbb{C}.
\end{equation}
\end{lemma}

Assume that $\psi^1$ and $\psi^2$ are two solutions of \eqref{eq:logNLS-limit} constructed in the proof, with the same initial data. Set
\[
\psi:=\psi^1-\psi^2 \quad\text{ and }\;\psi_0(x)=0,
\]
Then $\psi$ satisfies
\begin{equation}
i\partial_t\psi+\Delta\psi
=
\bigl(\ln|\psi^1|^2\,\psi^1-\ln|\psi^2|^2\,\psi^2\bigr)+\psi.
\end{equation}
Using the regularity of $\psi^1$ and $\psi^2$, we may perform an $L^2$ energy estimate. We obtain
\begin{align}
\frac12\frac{d}{dt}\|\psi(t)\|_{L^2(\mathbb{R}^d)}^2
&=
\operatorname{Im}\int_{\mathbb{T}^3}
\bigl(\ln|\psi^1|^2\,\psi^1-\ln|\psi^2|^2\,\psi^2\bigr)\overline{\psi}(t,x)\,dx
+\operatorname{Im}\int_{\mathbb{T}^3}\psi\,\overline{\psi}\,dx \notag\\
&=
\operatorname{Im}\int_{\mathbb{T}^3}
\bigl(\ln|\psi^1|^2\,\psi^1-\ln|\psi^2|^2\,\psi^2\bigr)
(\overline{\psi^1}-\overline{\psi^2})(t,x)\,dx \notag\\
&\leq C\,\|\psi(t)\|_{L^2(\mathbb{T}^3)}^2,
\end{align}
thanks to Lemma~\ref{control of complex}, where $C>0$ is the constant given there. By Gronwall's lemma and the fact that $\psi(0)=0$, it follows that $\psi\equiv 0$. Hence the solution is unique.

\end{proof}

The next step is to show the weak solution conserves the energy, which is based on the following lemma.

\begin{lemma}
\label{lemma:local-nonlinear-energy-convergence}
Assume
\[
\psi^{\delta}\xrightarrow[\delta\to 0]{}
\psi
\quad\text{strongly in }L^{2}\!\bigl((0,T)\times \mathbb{T}^3\bigr),
\]
and that there exists \(C_{T}>0\) such that
\[
\sup_{\delta>0}\|\psi^{\delta}\|_{L^{\infty}\!\bigl(0,T;H^{1}(\mathbb{T}^3)\bigr)}\le C_{T}.
\]
Define for \(\rho\ge 0\)
\[
f_{\delta}(\rho):=(\rho+\delta)\log(\rho+\delta),\qquad
f(\rho):=\rho\log\rho .
\]
Then, for almost every \(t\in(0,T)\),
\[
\lim_{\delta\to0}\int_{\mathbb{T}^3}
f_{\delta}\!\bigl(|\psi^{\delta}(t,\cdot)|^{2}\bigr)\,dx
=
\int_{\mathbb{T}^3}
f\!\bigl(|\psi(t,\cdot)|^{2}\bigr)\,dx .
\]
\end{lemma}

\begin{proof}
Because \(\psi^{\delta}\to\psi\) in
\(L^{2}\bigl((0,T)\times \mathbb{T}^3\bigr)\), up to a subsequence we have
\[
\psi^{\delta}(t,x)\longrightarrow\psi(t,x)
\qquad\text{for a.e. }(t,x)\in(0,T)\times \mathbb{T}^3.
\]
Hence
\(|\psi^{\delta}(t,x)|^{2}\to|\psi(t,x)|^{2}\) a.e. and,
since \(f_{\delta}\) is continuous in \(\rho\),
\[
f_{\delta}\!\bigl(|\psi^{\delta}|^{2}\bigr)
\longrightarrow
f\!\bigl(|\psi|^{2}\bigr)
\quad\text{a.e. on }(0,T)\times \mathbb{T}^3.
\]

For \(0<\delta\le1\) and \(\rho\ge0\) one has the elementary estimate
\[
(\rho+\delta)\log(\rho+\delta)
\;\le\;
C\,\bigl(1+\rho+\rho^{3}\bigr),
\]
which yields
\begin{equation}
\label{right hand of f}
\bigl|f_{\delta}\!\bigl(|\psi^{\delta}|^{2}\bigr)\bigr|
\le
C\,\bigl(1+|\psi^{\delta}|^{2}+|\psi^{\delta}|^{6}\bigr)
\qquad\text{on }(0,T)\times \mathbb{T}^3.
\end{equation}
The uniform bound in
\(L^{\infty}\!\bigl(0,T;H^{1}(\mathbb{T}^3)\bigr)\) implies, by Sobolev
embedding, a uniform bound in
\(L^{\infty}\!\bigl(0,T;L^{6}(\mathbb{T}^3)\bigr)\).
Consequently the right–hand side of \eqref{right hand of f}  is uniformly integrable on
\((0,T)\times \mathbb{T}^3\).

Thus we have a.e. convergence together with uniform integrability;
Vitali’s convergence theorem gives
\[
f_{\delta}\!\bigl(|\psi^{\delta}|^{2}\bigr)
\;\longrightarrow\;
f\!\bigl(|\psi|^{2}\bigr)
\quad\text{in }L^{1}\!\bigl((0,T)\times \mathbb{T}^3\bigr).
\]
Equivalently, for a.e. \(t\in(0,T)\)
\[
\int_{\mathbb{T}^3}
f_{\delta}\!\bigl(|\psi^{\delta}(t,\cdot)|^{2}\bigr)\,dx
\xrightarrow[\delta\to0]{}
\int_{\mathbb{T}^3}
f\!\bigl(|\psi(t,\cdot)|^{2}\bigr)\,dx .
\]
\end{proof}


\begin{proposition}
\label{proposition-energy equality}
If   \(\psi\) is a unique weak solution of \eqref{eq:logNLS-limit}, then it conserves the energy in the following sense: 
\begin{equation}
\label{energy equality for Sho}
E(\psi(t))=E(\psi(0))\;\;\;\text{ for any } t\geq 0,\end{equation}
where 
\[
E(\psi(t))
:=
\frac{\hbar^2}{2}\int_{\mathbb{T}^3}|\nabla\psi(t,x)|^2\,dx
+
\int_{\mathbb{T}^3}|\psi(t,x)|^2\log|\psi(t,x)|^2\,dx,
\]
 and therefore
\begin{equation}\label{eq:local-kinetic-conv}
\int_0^T\int_{\mathbb{T}^3}|\nabla\psi^\delta|^2\,dx\,dt
\to
\int_0^T\int_{\mathbb{T}^3}|\nabla\psi|^2\,dx\,dt.
\end{equation}

\end{proposition}

\begin{proof}
Let $\{\psi^\delta\}_{\delta>0}$ be a family of smooth solutions to the regularized problem
with  initial data $\psi^\delta(0)\to \psi(0)$ in $L^2(\mathbb{T}^3)$, and such that
\[
\psi^\delta \to \psi \quad \text{in } L^2(\mathbb{T}^3) \ \text{and a.e.}
\]

\medskip

\noindent
\textbf{Step 1. Energy identity at the approximate level.}
For each $\delta>0$, the solution $\psi^\delta$ is smooth, and a standard computation yields
\[
E(\psi^\delta(t))
=
E(\psi^\delta(0))
\quad \text{for all } t\ge 0,
\]
where
\[
E(\psi^\delta(t))
:=
\frac{\hbar^2}{2}\int_{\mathbb{T}^3} |\nabla \psi^\delta(t,x)|^2\,dx
+
\int_{\mathbb{T}^3}( |\psi^\delta(t,x)|^2+\delta) \log\big(|\psi^\delta(t,x)|^2+\delta\big)\,dx.
\]
\medskip

\noindent
\textbf{Step 2. Passage to the limit.} We pass to the limit $\delta\to 0$ to recover the energy inequality:

\medskip

\noindent
\emph{(i) Kinetic term.}
By weak convergence $\nabla \psi^\delta \rightharpoonup \nabla \psi$ in $L^2(\mathbb{T}^3)$ and lower semicontinuity,
\[
\int_{\mathbb{T}^3 }|\nabla \psi(t,x)|^2\,dx
\leq
\liminf_{\delta\to 0}
\int_{\mathbb{T}^3} |\nabla \psi^\delta(t,x)|^2\,dx.
\]

\medskip

\noindent
\emph{(ii) Logarithmic term.}
By Lemma \ref{lemma:local-nonlinear-energy-convergence},  one obtains that 
\[
\int_{\mathbb{T}^3} (|\psi^\delta(t,x)|^2+\delta) \log\big(|\psi^\delta(t,x)|^2+\delta\big)\,dx\to \int_{\mathbb{T}^3} |\psi(t,x)|^2 \log |\psi(t,x)|^2\,dx.
\]

\medskip

\noindent
\emph{(iii) Initial data.}
Similarly, we can show that \[
\int_{\mathbb{T}^3} (|\psi^\delta_{0}|^2+\delta) \log\big(|\psi^\delta_{0}|^2+\delta\big)\,dx\to \int_{\mathbb{T}^3}|\psi_0|^2 \log |\psi_0|^2\,dx.
\]

Thus, 
$
E(\psi^\delta(0))
\to
E(\psi(0)).
$

\medskip

\noindent
\textbf{Step 3. Energy inequality.} Combining the above estimates, we obtain
\begin{equation}
\label{energy inequality}
E(\psi(t))\leq 
E(\psi(0)),\quad\text{ for all} \,t\geq 0.
\end{equation}
\medskip

\noindent
\textbf{Step 4. Energy equality.}
Fix any $s\in [0,T]$ and define
\[
w(t,x):=\overline{\psi(s-t,x)}, \qquad t\in [0,s].
\]
To finish this step, we need the following lemma on $w(t,x)$.
\begin{lemma}[Time-reversal invariance]
\label{Time-reversal invariance}
Let $\psi$ be a weak solution of equation \eqref{eq:logNLS-limit}
in the sense of \eqref{eq:weak-form-limit}.  Fix $s\in (0,T]$ and define
\[
w(t,x):=\overline{\psi(s-t,x)}, \qquad (t,x)\in [0,s]\times \mathbb{T}^3.
\]
Then $w$ is a weak solution of  equation \eqref{eq:logNLS-limit}
 on $(0,s)\times \mathbb{T}^3$, with initial datum
\[
w(0,x)=\overline{\psi(s,x)}.
\]
More precisely, for every $\varphi\in C_c^\infty([0,s)\times \mathbb{T}^3)$,
\[
\int_0^s \int_{\mathbb{T}^3}
\left(
-i\hbar w\,\partial_t\varphi
-\frac{\hbar^2}{2}\nabla w\cdot \nabla \varphi
-(\log |w|^2+1)w\,\varphi
\right)\,dx\,dt
=
i\hbar \int_{\mathbb{T}^3} \overline{\psi(s,x)}\,\varphi(0,x)\,dx.
\]
\end{lemma}

\begin{proof}
Let $\varphi\in C_c^\infty([0,s)\times \mathbb{T}^3)$ and define
\[
\widetilde{\varphi}(\tau,x):=\overline{\varphi(s-\tau,x)},
\qquad (\tau,x)\in [0,s]\times \mathbb{T}^3.
\]
Then $\widetilde{\varphi}\in C_c^\infty([0,s)\times \mathbb{T}^3)$ and, since
$\operatorname{supp}\varphi \subset [0,s-\varepsilon]\times \mathbb{T}^3$ for some $\varepsilon>0$,
we have $\widetilde{\varphi}(0,x)=0$.

We apply the weak formulation of  \eqref{eq:weak-form-limit} for $\psi$ on $(0,s)$ with test function
$\widetilde{\varphi}$:
\[
\int_0^s \int_{\mathbb{T}^3}
\left(
-i\hbar \psi\,\partial_\tau \widetilde{\varphi}
-\frac{\hbar^2}{2}\nabla \psi\cdot \nabla \widetilde{\varphi}
-(\log |\psi|^2+1)\psi\,\widetilde{\varphi}
\right)\,dx\,d\tau
=0.
\]

Taking complex conjugates and performing the change of variables $t=s-\tau$, we obtain
\[
\int_0^s \int_{\mathbb{T}^3}
\left(
-i\hbar \overline{\psi(s-t,x)}\,\partial_t\varphi
-\frac{\hbar^2}{2}\nabla \overline{\psi(s-t,x)}\cdot \nabla \varphi
-(\log |\psi(s-t,x)|^2+1)\overline{\psi(s-t,x)}\,\varphi
\right)\,dx\,dt
\]
\[
=
i\hbar \int_{\mathbb{T}^3} \overline{\psi(s,x)}\,\varphi(0,x)\,dx.
\]

Since $w(t,x)=\overline{\psi(s-t,x)}$ and $|w|=|\psi(s-t)|$, this is exactly the weak formulation
for $w$ on $(0,s)\times \mathbb{R}^3$ with initial datum $w(0,x)=\overline{\psi(s,x)}$.
\end{proof}

By Lemma \ref{Time-reversal invariance},
 $w$ is a weak solution of  \eqref{eq:logNLS-limit} on $[0,s]$ with initial datum
$w(0)=\overline{\psi(s)}$. Applying the energy inequality \eqref{energy inequality} to $w$, we obtain
\[
E(w(s)) \le E(w(0)).
\]
Since $w(s)=\overline{\psi(0)}$ and $w(0)=\overline{\psi(s)}$, and since the energy depends only on
$|\psi|$ and $|\nabla \psi|$, we have
\[
E(\psi(0)) = E(w(s)) \le E(w(0)) = E(\psi(s)).
\]
On the other hand, applying \eqref{energy inequality} directly to $\psi$, we also have
\[
E(\psi(s)) \le E(\psi(0)).
\]
Therefore
\[
E(\psi(s)) = E(\psi(0)) \qquad \text{for all } s\in [0,T],
\]
hence the energy is conserved on $[0,T]$. Combining this identity with Lemma~\ref{lemma:local-nonlinear-energy-convergence}, we obtain
\[
\int_{\mathbb{T}^3} |\nabla\psi^\delta(t,x)|^2\,dx \to \int_{\mathbb{T}^3} |\nabla\psi(t,x)|^2\,dx
\quad \text{for a.e. } t\in(0,T).
\]
Moreover, we have the uniform bound
\[
\sup_{\delta>0}\sup_{t\in[0,T]} \int_{\mathbb{T}^3} |\nabla\psi^\delta(t,x)|^2\,dx \le C_{T,K} < \infty.
\]
Thus, defining \(f^\delta(t) := \int_{\mathbb{T}^3} |\nabla\psi^\delta(t,x)|^2\,dx\), the functions \(f^\delta\) converge pointwise almost everywhere to \(f(t) := \int_{\mathbb{T}^3} |\nabla\psi(t,x)|^2\,dx\) and are dominated by the integrable constant \(C_{T}\). By the dominated convergence theorem,
\[
\int_0^T\int_{\mathbb{T}^3} |\nabla\psi^\delta(t,x)|^2\,dx\,dt
\to
\int_0^T\int_{\mathbb{T}^3} |\nabla\psi(t,x)|^2\,dx\,dt.
\]

\end{proof}


\section{Strong convergence of $\nabla\psi^{\delta}$ in $L^2$}

Our goal of this section is to show, up to a subsequence,
\[
\nabla\psi^\delta \to \nabla\psi
\qquad\text{strongly in }L^2(0,T;L^2( \mathbb{T}^3)),
\]
for every \(T>0\).

\begin{proposition}[Strong convergence of the gradients]
\label{prop:strong-convergence-gradient}
Under the assumptions of Proposition~\ref{proposition-energy equality}, up to extraction of a subsequence,
\[
\nabla \psi^\delta \to \nabla \psi
\qquad\text{strongly in } L^2(0,T;L^2( \mathbb{T}^3)).
\]
 Equivalently, the sequence \(\{\nabla\psi^\delta\}\) is relatively compact in
\[
L^2\bigl([0,T];L^2(\mathbb{T}^3)\bigr).
\]
In particular,
\[
\nabla\sqrt{\rho^\delta}\to \nabla\sqrt{\rho}
\qquad\text{strongly in }L^2(0,T;L^2( \mathbb{T}^3)),
\]
and
\[
\Lambda^\delta\to \Lambda
\qquad\text{strongly in } L^2(0,T;L^2( \mathbb{T}^3)).
\]
\end{proposition}

\begin{proof}
By Proposition~\ref{prop:compactness} we have 
\[
\nabla\psi^\delta \rightharpoonup \nabla\psi
\qquad\text{weakly in } L^2(0,T;L^2( \mathbb{T}^3)).
\]
On the other hand, by \eqref{eq:local-kinetic-conv} in Proposition~\ref{proposition-energy equality},  we have 
\[
\int_0^T\int_{\mathbb{T}^3} |\nabla\psi^\delta|^2\,dx\,dt
\to
\int_0^T\int_{\mathbb{T}^3} |\nabla\psi|^2\,dx\,dt.
\]
Since \(L^2(0,T;L^2( \mathbb{T}^3))\)  is a Hilbert space, weak convergence together with convergence of norms implies
\[
\nabla\psi^\delta \to \nabla\psi
\qquad\text{strongly in }L^2((0,T)\times \mathbb{T}^3).
\]
The strong convergence of \(\nabla\sqrt{\rho^\delta}\) and \(\Lambda^\delta\) then follows from Lemma 3 in \cite{AntonelliMarcati2009}.

\end{proof}

\section{The weak formulation induced by the Schr\"odinger approximation}
\label{sec:induced}

In this section, we derive the weak formulation of the quantum hydrodynamic system induced by the Schr\"odinger approximation, and pass to the limit by means of the compactness results established in Proposition~\ref{prop:compactness} and Proposition~\ref{prop:strong-convergence-gradient}. This completes the proof of the main theorem.

\subsection{Mollified approximation and induced hydrodynamic variables}

Let \(\eta\in C_c^\infty(\mathbb{R}\times \mathbb{T}^3)\) be a standard nonnegative mollifier with
$
\int_{\mathbb{R}\times\mathbb{T}^3}\eta=1,$
and define
\[
\eta_\varepsilon(t,x)=\varepsilon^{-4}\eta(t/\varepsilon,x/\varepsilon).
\]
Extending \(\psi^\delta\) by \(0\) outside \([0,T]\), we set
\[
\psi^\delta_\varepsilon:=\eta_\varepsilon * \psi^\delta.
\]
Then \(\psi^\delta_\varepsilon\in C^\infty(\mathbb{R}\times\mathbb{T}^3)\) and satisfies the mollified equation
\begin{equation}\label{eq:moll-eq-rig}
i\hbar\,\partial_t\psi^\delta_\varepsilon+\frac{\hbar^2}{2}\Delta\psi^\delta_\varepsilon
=
\eta_\varepsilon * F_\delta(\psi^\delta).
\end{equation}
Define the associated hydrodynamic quantities
\[
\rho^\delta_\varepsilon:=|\psi^\delta_\varepsilon|^2,
\qquad
J^\delta_\varepsilon:=\hbar\,\Im\!\bigl(\psi^\delta_\varepsilon\nabla\overline{\psi^\delta_\varepsilon}\bigr),
\]
and introduce the commutator
\begin{equation}\label{eq:Gepsdelta-rig}
G_{\varepsilon,\delta}
:=
\eta_\varepsilon * F_\delta(\psi^\delta)-F_\delta(\psi^\delta_\varepsilon).
\end{equation}

\medskip

A direct computation gives the following continuity equation with commutator remainder.

\begin{proposition}\label{prop:continuity-eps-delta}
The quantities \((\rho^\delta_\varepsilon,J^\delta_\varepsilon)\) satisfy
\[
\partial_t\rho^\delta_\varepsilon+\operatorname{div}J^\delta_\varepsilon
=
R_{\varepsilon,\delta}^{(0)},
\]
where
\[
R_{\varepsilon,\delta}^{(0)}
:=
\frac{2}{\hbar}\,\Im\!\bigl(\overline{\psi^\delta_\varepsilon}G_{\varepsilon,\delta}\bigr).
\]
Equivalently, for every \(\zeta\in C_c^\infty([0,T)\times\mathbb{T}^3)\),
\begin{equation}\label{eq:weak-with-R0}
\int_0^T\!\!\int_{\mathbb{T}^3}
\Bigl(
\rho^\delta_\varepsilon\,\partial_t\zeta
+
J^\delta_\varepsilon\cdot\nabla\zeta
-
R^{(0)}_{\varepsilon,\delta}\zeta
\Bigr)\,dx\,dt
+
\int_{\mathbb{T}^3}\rho^\delta_\varepsilon(0,x)\zeta(0,x)\,dx
=
0.
\end{equation}
\end{proposition}

\medskip

Next, we derive the momentum identity with an commutator remainder in stress form.

\begin{proposition}\label{prop:momentum-eps-delta}
The quantities $(\rho^\delta_\varepsilon,J^\delta_\varepsilon)$ satisfy
\begin{equation}\label{eq:mom-stress-eps-delta}
\partial_t J^\delta_\varepsilon+\nabla P_\delta(\rho^\delta_\varepsilon)
=
-\hbar^2\operatorname{div}\Re\!\bigl(\nabla\psi^\delta_\varepsilon\otimes\nabla\overline{\psi^\delta_\varepsilon}\bigr)
+\frac{\hbar^2}{4}\Delta\nabla\rho^\delta_\varepsilon
+R^{(1)}_{\varepsilon,\delta},
\end{equation}
where
\[
R^{(1)}_{\varepsilon,\delta}
=
\Re\!\bigl(G_{\varepsilon,\delta}\nabla\overline{\psi^\delta_\varepsilon}
-
\overline{\psi^\delta_\varepsilon}\nabla G_{\varepsilon,\delta}\bigr).
\]
Moreover, after polar decomposition, one obtains the weak formulation
\begin{equation}\label{eq:weak-qhd-momentum-eps-delta}
\begin{aligned}
\int_0^T\!\!\int_{\mathbb{T}^3}
\Big(
J^\delta_\varepsilon\cdot \partial_t\varphi
+
\Lambda^\delta_\varepsilon\otimes\Lambda^\delta_\varepsilon : \nabla\varphi
+
P_\delta(\rho^\delta_\varepsilon)\,\operatorname{div}\varphi
+
\hbar^2\,\nabla\sqrt{\rho^\delta_\varepsilon}\otimes\nabla\sqrt{\rho^\delta_\varepsilon}:\nabla\varphi
\\
-
\frac{\hbar^2}{4}\rho^\delta_\varepsilon\,\Delta(\operatorname{div}\varphi)
-
R^{(1)}_{\varepsilon,\delta}\cdot\varphi
\Big)\,dx\,dt
+
\int_{\mathbb{T}^3}
J^\delta_\varepsilon(0,x)\cdot\varphi(0,x)\,dx
=0
\end{aligned}
\end{equation}
for every $\varphi\in C_c^\infty([0,T)\times\mathbb{T}^3;\mathbb{R}^3)$.
\end{proposition}

\begin{proof}
Differentiating
\[
J^\delta_\varepsilon=\hbar\,\Im(\overline{\psi^\delta_\varepsilon}\nabla\psi^\delta_\varepsilon)
\]
and using \eqref{eq:moll-eq-rig}, we obtain
\[
\partial_t J^\delta_\varepsilon + \nabla P_\delta(\rho^\delta_\varepsilon)
=
-\hbar^2\,\operatorname{div}\Re\!\bigl(\nabla\psi^\delta_\varepsilon\otimes\nabla\overline{\psi^\delta_\varepsilon}\bigr)
+\frac{\hbar^2}{4}\Delta\nabla\rho^\delta_\varepsilon
+R^{(1)}_{\varepsilon,\delta},
\]
where
\[
R^{(1)}_{\varepsilon,\delta}
=
\Re\!\bigl(G_{\varepsilon,\delta}\nabla\overline{\psi^\delta_\varepsilon}
-
\overline{\psi^\delta_\varepsilon}\nabla G_{\varepsilon,\delta}\bigr).
\]
Applying the polar decomposition lemma (Lemma~3 in \cite{AntonelliMarcati2009}) to $\psi^\delta_\varepsilon$, we have
\[
\hbar^2\,\Re\!\bigl(\nabla\psi^\delta_\varepsilon\otimes\nabla\overline{\psi^\delta_\varepsilon}\bigr)
=
\hbar^2\,\nabla\sqrt{\rho^\delta_\varepsilon}\otimes\nabla\sqrt{\rho^\delta_\varepsilon}
+
\Lambda^\delta_\varepsilon\otimes\Lambda^\delta_\varepsilon,
\]
where
\[
J^\delta_\varepsilon=\sqrt{\rho^\delta_\varepsilon}\,\Lambda^\delta_\varepsilon.
\]
Substituting this identity into \eqref{eq:mom-stress-eps-delta}, we get
\[
\partial_t J^\delta_\varepsilon
+\operatorname{div}\bigl(\Lambda^\delta_\varepsilon\otimes\Lambda^\delta_\varepsilon\bigr)
+\nabla P_\delta(\rho^\delta_\varepsilon)
=
-\hbar^2\,\operatorname{div}\!\bigl(\nabla\sqrt{\rho^\delta_\varepsilon}\otimes\nabla\sqrt{\rho^\delta_\varepsilon}\bigr)
+\frac{\hbar^2}{4}\Delta\nabla\rho^\delta_\varepsilon
+R^{(1)}_{\varepsilon,\delta}.
\]
Testing against $\varphi\in C_c^\infty([0,T)\times\mathbb{T}^3;\mathbb{R}^3)$ and integrating by parts in time and space, we obtain \eqref{eq:weak-qhd-momentum-eps-delta}.
\end{proof}

\subsection{Vanishing of the commutator remainders}
In this subsection, we develop the following lemma for weak vanishing of the commutator remainders. This allows us to control the two commutator terms in the weak formulation. 

\begin{lemma}[Weak vanishing of the commutator remainders]
\label{Lemma commutator}
Let $0<T<\infty$ and $\psi\in L^\infty(0,T;H^1(\mathbb{T}^3))$.  
Let $\delta=\varepsilon\to 0$  and define
\[
G_\varepsilon := G_{\varepsilon,\delta(\varepsilon)},\qquad
R^{(j)}_\varepsilon := R^{(j)}_{\varepsilon,\delta(\varepsilon)},\quad j=0,1.
\]
Then
\begin{equation}\label{eq:G-assumptions-log}
G_\varepsilon \to 0 \quad \text{in } L^2((0,T)\times\mathbb{T}^3) \text{ as } \varepsilon \to 0,
\end{equation}
and that $\psi_\varepsilon$ and $\nabla\psi_\varepsilon$ are uniformly  bounded in $L^2$,
 and further 
\[
R^{(0)}_\varepsilon \to 0 \quad \text{and} \quad R^{(1)}_\varepsilon \to 0 \quad \text{in } \mathcal D'((0,T)\times\mathbb{T}^3) \text{ as } \varepsilon \to 0.
\]
\end{lemma}

\begin{proof}
We first establish the bound for $G_\varepsilon$.  
From the gradient estimate of the logarithmic nonlinearity,
\[
|\nabla F_\delta(\psi)| \le C \big(1 + |\log(|\psi|^2+\delta)|\big)\,|\nabla\psi|,
\]
and the Lipschitz property
\[
|F_\delta(z_1) - F_\delta(z_2)| \le C(1 + |\log \delta|) |z_1 - z_2|.
\]

Let us recall
 the commutator
\[
G_{\varepsilon,\delta} := \eta_\varepsilon * F_\delta(\psi) - F_\delta(\psi_\varepsilon),
\]
which we split as
\begin{equation}\label{eq:G-split-log}
G_{\varepsilon,\delta} = A_{\varepsilon,\delta} + B_{\varepsilon,\delta}, \quad
A_{\varepsilon,\delta} := \eta_\varepsilon * F_\delta(\psi) - F_\delta(\psi),\quad
B_{\varepsilon,\delta} := F_\delta(\psi) - F_\delta(\psi_\varepsilon).
\end{equation}

Using the standard mollifier estimate,
\[
\|\eta_\varepsilon * g - g\|_{L^2} \le C \varepsilon \|\nabla g\|_{L^2},
\]
we deduce
\[
\|A_\varepsilon\|_{L^2} \le C \varepsilon (1 + |\log \delta(\varepsilon)|)\|\nabla \psi\|_{L^2},\qquad
\|B_\varepsilon\|_{L^2} \le C \varepsilon (1 + |\log \delta(\varepsilon)|)\|\nabla \psi\|_{L^2}.
\]

Choosing $\varepsilon=\delta(\varepsilon) $ ensures $\varepsilon(1 + |\log \delta(\varepsilon)|) \to 0$, and hence
\[
G_\varepsilon = A_\varepsilon + B_\varepsilon \to 0 \quad \text{in } L^2((0,T)\times \mathbb{T}^3).
\]

\paragraph{Weak vanishing of $R^{(0)}_\varepsilon$.}  
For $\zeta \in C_0^\infty((0,T)\times\mathbb{T}^3)$, Hölder's inequality gives
\[
|\langle R^{(0)}_\varepsilon, \varphi \rangle|
\le \frac{2}{\hbar} \int |\psi_\varepsilon|\,|G_\varepsilon|\,|\zeta|
\le C_\zeta \|\psi_\varepsilon\|_{L^2(\mathrm{supp}\,\zeta)} \|G_\varepsilon\|_{L^2(\mathrm{supp}\,\zeta)} \to 0.
\]

\paragraph{Weak vanishing of $R^{(1)}_\varepsilon$.}  
For $\varphi \in C_0^\infty((0,T)\times\mathbb{T}^3;\mathbb{R}^3)$, integration by parts yields
\[
\langle R^{(1)}_\varepsilon, \varphi \rangle = \Re \int \big( G_\varepsilon \nabla \overline{\psi_\varepsilon} \cdot \varphi + \overline{\psi_\varepsilon} G_\varepsilon\, \mathrm{div}\,\varphi \big)\,dx\,dt,
\]
so that
\[
|\langle R^{(1)}_\varepsilon, \varphi \rangle| \le C_\varphi \big( \|\psi_\varepsilon\|_{H^1(\mathrm{supp}\,\varphi)} \big) \|G_\varepsilon\|_{L^2(\mathrm{supp}\,\varphi)} \to 0.
\]

This proves the lemma.
\end{proof}

\subsection{Passage to the limit and proof of the main theorem}

It remains to pass to the limit in the weak formulations \eqref{eq:weak-with-R0} and
\eqref{eq:weak-qhd-momentum-eps-delta}. By Proposition~\ref{prop:compactness}, we have
\[
\psi^\delta\to\psi
\qquad\text{strongly in }L^2((0,T)\times\mathbb{T}^3),
\]
hence
\[
\sqrt{\rho^\delta}=|\psi^\delta|\to|\psi|=\sqrt{\rho}
\qquad\text{strongly in }L^2((0,T)\times\mathbb{T}^3),
\]
and therefore
\[
\rho^\delta\to\rho
\qquad\text{strongly in }L^1((0,T)\times\mathbb{T}^3).
\]
Moreover, by Proposition~\ref{prop:strong-convergence-gradient},
\[
\nabla\sqrt{\rho^\delta}\to\nabla\sqrt{\rho},
\qquad
\Lambda^\delta\to\Lambda
\qquad\text{strongly in }L^2((0,T)\times\mathbb{T}^3),
\]
so that
\[
J^\delta=\sqrt{\rho^\delta}\Lambda^\delta\to \sqrt{\rho}\Lambda=J
\qquad\text{strongly in }L^1((0,T)\times\mathbb{T}^3).
\]
Similarly,
\[
\Lambda^\delta\otimes\Lambda^\delta\to \Lambda\otimes\Lambda,
\qquad
\nabla\sqrt{\rho^\delta}\otimes\nabla\sqrt{\rho^\delta}
\to
\nabla\sqrt{\rho}\otimes\nabla\sqrt{\rho}
\qquad\text{strongly in }L^1((0,T)\times\mathbb{T}^3),
\]
and
\[
P_\delta(\rho^\delta)\to \rho
\qquad\text{strongly in }L^1((0,T)\times\mathbb{T}^3)
\]
by the pointwise convergence of \(P_\delta(\rho)=\rho-\delta\log(\rho+\delta)\) and the strong convergence of \(\rho^\delta\).

Passing to the limit in \eqref{eq:weak-with-R0} and \eqref{eq:weak-qhd-momentum-eps-delta}, and using the vanishing of the commutator remainders in Lemma \ref{Lemma commutator}, we recover the weak formulation of the collisionless quantum Euler system for the limiting pair \((\rho,J)\).

Meanwhile,  using \eqref{eq:quadratic-identity} to \eqref{energy equality for Sho}, one obtains that the energy equality for weak solution $(\rho,J)$:
\begin{equation*}
\begin{split}
E(t)
&=
\int_{\mathbb{T}^3}
\left(
\frac{\hbar^2}{2}|\nabla\sqrt{\rho}(t,x)|^2
+
\frac12 |\Lambda(t,x)|^2
+
\rho(t,x)\log \rho(t,x)
\right)\,dx
\\&=\int_{\mathbb{T}^3}
\left(
\frac{\hbar^2}{2}|\nabla\sqrt{\rho_0}|^2
+
\frac12 |\Lambda_0|^2
+
\rho_0\log \rho_0
\right)\,dx
=E(0)
\end{split}
\end{equation*}
for any time $t>0$.

This completes the proof of the main theorem.

  \section*{Acknowledgments}

The author  is partially supported
by the NSF grant: DMS-2510425, and by the Simons Foundation: MPS-TSM-00007824.

\end{document}